\documentclass[12pt]{article}

\usepackage{amssymb}
\usepackage{graphicx}
 \usepackage{amsmath}
 \usepackage{amsthm}

\usepackage{authblk}
\usepackage{color}
\newtheorem{theorem}{Theorem}
\usepackage{caption}
\usepackage{subcaption}
\usepackage{natbib}

\newtheorem{lem}[theorem]{Lemma}
\newtheorem{conj}{Conjecture}
\newtheorem{mydef}{Definition}

\title{Isoperimetric bubbles in spaces with density $r^p + a$ }

\author{Martyn Gwynne\thanks{mag117@aber.ac.uk}\,}
\author{Simon Cox\thanks{sxc@aber.ac.uk} }
\affil{Department of Mathematics, Aberystwyth University,\\ Ceredigion SY23 3BZ, UK}

\begin{document}

\maketitle

\begin{abstract}
    Least perimeter solutions for a region with fixed mass are sought in ${\mathbb{R}^d}$ on which a density function $\rho(r) = r^p+a$, with $p>0, a>0$, weights both perimeter and mass.
    On the real line ($d=1$) this is a single interval that includes the origin. For $p \le 1$ the isoperimetric interval has one end at the origin; for larger $p$ there is a critical value of $a$ above which the interval is symmetric about the origin. In the case $p=2$, for $d=2$ and $3$, the isoperimetric region is a circle or sphere, respectively, that includes the origin; the centre moves towards the origin as $a$ increases, with constant radius, and then remains centred on the origin for $a$ above the critical value as the radius decreases.
\end{abstract}

\section{Introduction}  

A bubble floating in three-dimensional (3D) Euclidean space assumes the shape of a sphere because its surface area is minimised for a given volume. This is an example of an isoperimetric problem~\citep{morgan2016geometric}. Here, we seek to describe the shape of a bubble in spaces where the density is not constant, and instead increases with distance from the origin, where it is non-zero. 

Both the area and volume of the ``bubble" are weighted by the density $\rho(r)$. In terms of the underlying Riemannian volume $dV$ and perimeter $dP$ in 3D, the new weighted volume  and surface area are given by
$dV_{\rho}=\rho(r) dV$ and
$dA_{\rho}=\rho(r) dA$.

In 2D we consider isoperimetric regions in the plane, that is, regions that minimize their weighted perimeter with fixed, weighted, area. And in 1D isoperimetric regions are line intervals that minimize the weighted position of their endpoints subject to a fixed, weighted, length.

\cite{morgan2013existence} show that if $\rho$ is a radial, increasing density on $\mathbb{R}$, then isoperimetric sets exist for any bubble volume. The log-convex density theorem~\citep {chambers2019proof} implies that balls around the origin are isoperimetric if the density function increases rapidly away from the origin (to be precise, if $\rho(r)={\rm e}^{g(r)}$ with $g$ smooth, convex and even; that is, if $\rho(r)$ is a log-convex function), and that this is the unique solution. 

On the other hand, isoperimetric regions in 1D  (i.e. intervals on the real line $x \in (-\infty,+\infty)$) for densities that increase with distance from the origin are connected and have one end at the origin~\citep{rosales2008isoperimetric,huang2019isoperimetric}. In 2D, in the plane with radial density $r^p$, $p>0$, an isoperimetric region is a circle through the origin (that is, its perimeter passes through the origin)~\citep{dahlberg2010isoperimetric}. 

We seek to reconcile these two results with a density function that has a free (additive) parameter, $\rho(r) = r^p+a$, that allows us to interpolate between densities of the form $r^p$ and those that are log-convex for sufficiently large $a$.\footnote{A further motivation for this choice of density function is the desire to avoid the physically questionable situation of having zero density at a point in space.} Given the  mass of the region, i.e. its weighted volume (which may be the length of an interval in $\mathbb{R}^1$, the area inside a closed contour in $\mathbb{R}^2$ or the volume inside a closed surface in $\mathbb{R}^3$), we try to describe the properties of an isoperimetric interval for different values of the power $p$ and offset $a$, with $p> 0$ and $a>0$.

\section{When the density is log-convex}
\label{sec:logconvex}

We determine the conditions under which the density function $\rho(r) = r^p+a$ is a log-convex function of $r$.

\begin{mydef}
A function $\rho(r)$ is \emph{log-convex} if $\psi(r)=\log(\rho(r))$ has positive second derivative, $\psi^{\prime\prime}>0$.
\end{mydef}

 
\begin{lem} 
\label{lem:logconvex_allD}
Let the density function be $\rho(r)=r^p+a$, with $p>1$ and $a>0$. A symmetric region of mass $M$ in $d$ dimensions, centred on the origin, is isoperimetric if $a$ is greater than
\begin{equation}
a_{crit}=
  \left(\frac{d(p+d)}{k_d p(d+1)}
   \right)^\frac{p}{p+d}
   (p-1)^{-\frac{d}{p+d}}
    M^\frac{p}{p+d},
\label{eq:acrit_p_gt_1_allD} 
\end{equation} 
where $k_d$ is the surface area of a $d$-dimensional unit sphere, $k_1= 1$, $k_2 = 2\pi$, $k_3 = 4\pi$.
\end{lem}

So if this value of $a$ is exceeded, a 1D isoperimetric region is an interval of the real line that is symmetric about the origin. In 2D, the isoperimetric region is a circle centred on the origin. And in 3D it is a sphere with centre at the origin.

\begin{proof}
We have
\begin{equation}
\psi^{\prime\prime} =  \frac{p \,r^{p-2}}{\left(r^p+a\right)^2}
\left( a(p-1)-r^p \right)
\label{eq:psi}
\end{equation}
so the density function is convex if $r^p < a(p-1)$. That is, provided the region can fit within a distance $R=\left( a(p-1) \right)^{1/p}$ of the origin, it will be a symmetric ``ball".

The mass of a $d$-dimensional ball with this critical radius $R = \left( a(p-1) \right)^{1/p}$,  centred on the origin, is 
 \begin{eqnarray}
  M_{crit} & =& k_d\int_0^R r^{d-1} \rho(r) \, {\rm d}r \nonumber \\
  & =&  k_d R^d\left(\frac{R^{p}}{p+d}+\frac{a}{d}\right) 
  \label{eq:2DMcrit_1} \\
  & = &\frac{k_d p(d+1)}{d(p+d)} (p-1)^\frac{d}{p}
    a^\frac{p+d}{p}.
\label{eq:2DMcrit}
\end{eqnarray}
Rearranging this expression gives the result for $a_{crit}$ in~eq.~(\ref{eq:acrit_p_gt_1_allD}).

\end{proof}

For larger values of the offset $a$ or, equivalently, smaller mass $M$, it is necessary to solve eq.~(\ref{eq:2DMcrit_1}) for $R$ in order to find the perimeter of the isoperimetric region. This is not always straightforward. The critical value of $a$ is not defined for $p < 1$. Hence for $p < 1$ the isoperimetric region is not a ball centred on the origin. We discuss below, in each of dimensions one to three, the details of the shape of the isoperimetric region for $a$ less than this critical value.

\section{Isoperimetric intervals on the real line (1D)}
\label{sec:defns}

We first consider a single interval on the real line. We show that the isoperimetric interval is connected, and that it includes the origin.

\begin{mydef} 
On the real line $\mathbb{R}$, a $\emph {density}$ function $\rho$ is a positive continuous function that weights each point $x$ in $\mathbb{R}$ with a certain mass $\rho(x)$.  Given an interval $[\alpha,\beta] \subset\mathbb{R}$, the weighted perimeter and mass of the interval $[\alpha,\beta]$ are given by:
\[ {\rm Perimeter, } \quad P= \rho(\alpha) + \rho(\beta)\]
and
\[ {\rm Mass,} \quad M=\int_{\alpha}^{\beta} \rho(x)  {\rm d}x.
\] 
\end{mydef}

\begin{mydef} 
An interval $[\alpha,\beta]\subset\mathbb{R}$ is \emph{isoperimetric} if it has the smallest weighted perimeter of all intervals with the same weighted mass. 
\end{mydef}

With the density function $\rho(x) = |x| ^p + a$  considered here, on an interval $[\alpha,\beta]$, we will show that the left-hand end of the interval ($\alpha$) is always non-positive. Hence 
\begin{equation}
   P(\alpha,\beta)= |\alpha|^{p}  + \beta^{p} + 2 a
   \label{eq:perim_general}
\end{equation}
and
\begin{equation}
   M(\alpha,\beta)= \frac{\beta^{p+1}}{p+1}+ \frac{|\alpha|^{p+1}}{p+1}  + (\beta +|\alpha|)a.
   \label{eq:mass_general}
\end{equation}

To ensure that we may consider just a single interval in seeking to determine the interval with the least weighted perimeter, we need to show that any set of $N$ distinct intervals with a given total mass can be reduced to a single interval with at most the same perimeter and the same mass. Moreover, as the left end of this interval approaches the origin from above, its perimeter decreases.

    \begin{lem}
     For the density function $\rho(x)=|x|^p + a$, $p>0$ , $a > 0,$\ any finite set of intervals $[\alpha_i,\beta_i]$, $1\leq i\leq N,$ can be reduced to a single interval with an equal or lower perimeter while keeping the mass constant.  This interval, denoted $[\alpha,\beta]$, includes the origin, with $\alpha \leq 0, \beta>0$.
     \label{lem:joinintervals}
    \end{lem}

The proof closely follows the one given in~\cite{huang2019isoperimetric} for $\rho(x) = |x|^p$; further details can be found in appendix~\ref{sec:app-twointervals}.

We now focus on a single interval $[\alpha,\beta]$ where $\alpha\leq0$ without loss of generality, because of the symmetry of $\rho(x)$. Figure~\ref{fig:sketch} shows examples of the density function $\rho(x)=| x | ^p + a$ and illustrates the different possible least perimeter intervals which we will derive. 

\begin{figure}
\centering
 \begin{subfigure}[b]{0.32\textwidth}
   \includegraphics[width=\textwidth]{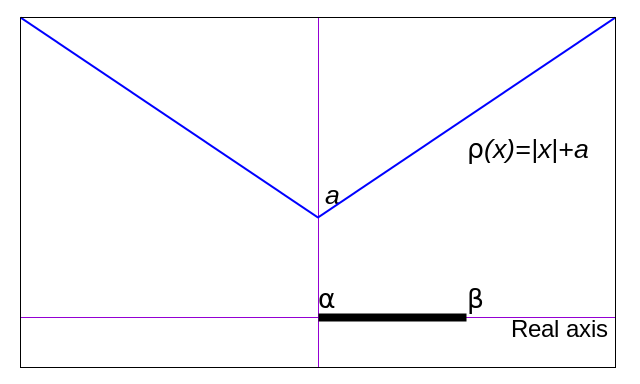}
   \caption{}
   \label{fig:sketchn1}
 \end{subfigure}
 \hfill
 \begin{subfigure}[b]{0.32\textwidth}
   \includegraphics[width=\textwidth]{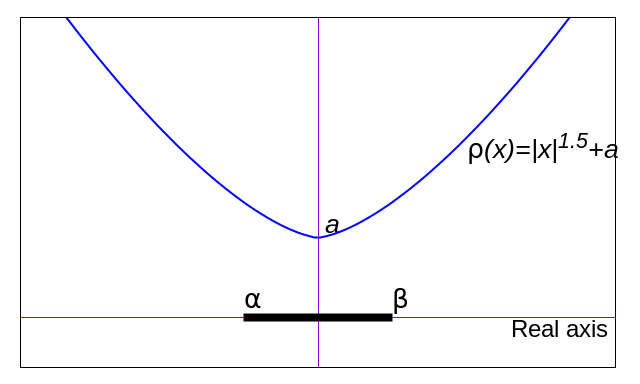}
   \caption{}
   \label{fig:sketchn1.5}
 \end{subfigure}
 \hfill
 \begin{subfigure}[b]{0.32\textwidth}
   \includegraphics[width=\textwidth]{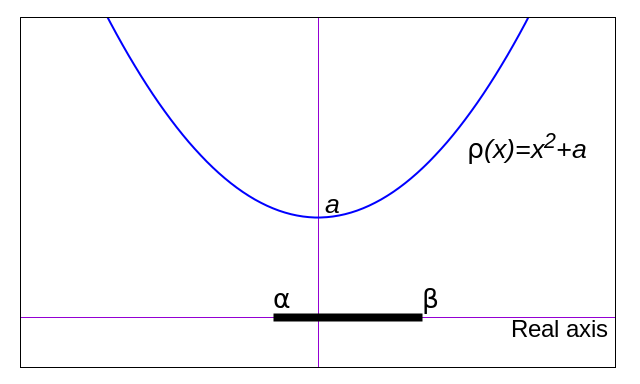}
   \caption{}
   \label{fig:sketchn2}
 \end{subfigure}
 \hfill    
   \caption{Examples of the 1D density function $\rho(x)=| x | ^p+a$ considered here and different least-perimeter intervals $[\alpha,\beta]$. (a) If $p=1$, the least perimeter interval always has one end at the origin. (b) For $p=1.5$, a symmetric interval is the least perimeter solution for large enough $a$, (c) For $p=2$, the isoperimetric solution is an asymmetric interval if $a$ is less than a critical value.}
\label{fig:sketch}
\end{figure}

The perimeter and extent of the optimal interval vary strongly with the offset $a$ and the exponent $p$. 
The values of $p$ and $a$ for which the isoperimetric interval is symmetric can be found by determining when $\rho$ is log-convex, and applying the log-convex density theorem, as evaluated for this density function in Lemma~\ref{lem:logconvex_allD}. That is, the interval is symmetric if $a$ is greater than
\begin{equation}
a_{crit}=\left(\frac{p+1}{4p}\right)^\frac{p}{p+1} (p-1)^{-\frac{1}{p+1}} M^\frac{p}{p+1}.
\label{eq:acrit_p_gt_1} 
\end{equation} 
Then to find the perimeter of the interval from eq.~(\ref{eq:perim_general}), with $|\alpha|=\beta$ (i.e. $P=2\beta^p+2a$), requires solving eq.~(\ref{eq:mass_general}) to find $\beta$ for a given mass $M_0$ from $M_0 = 2\beta^{p+1}/(p+1) + 2\beta a$. For most values of $p$ this is not possible -- we give an example for $p=2$ below -- although we can substitute for $\beta$ in terms of the perimeter to give an implicit expression for the mass in terms of the perimeter for any $a$ and $p$.

For $a<a_{crit}$, where the log-convex density theorem does not apply, we can describe explicit formulae for the perimeter and end-points of the isoperimetric interval only in the cases $p=\frac{1}{2}$, $p=1$ and $p=2$. We validate our predictions with simulations using the Surface Evolver~\citep{brakke92}.

\begin{figure}
     \centering
     \begin{subfigure}[b]{0.48\textwidth}
         \includegraphics[width=\textwidth]{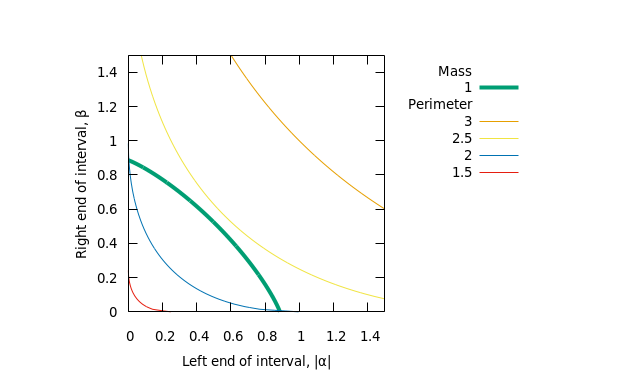}
         \caption{$p=0.5, a=0.5$}
         \label{fig:contourplots_n0.5}
     \end{subfigure}
     \hfill
     \begin{subfigure}[b]{0.48\textwidth}
         \includegraphics[width=\textwidth]{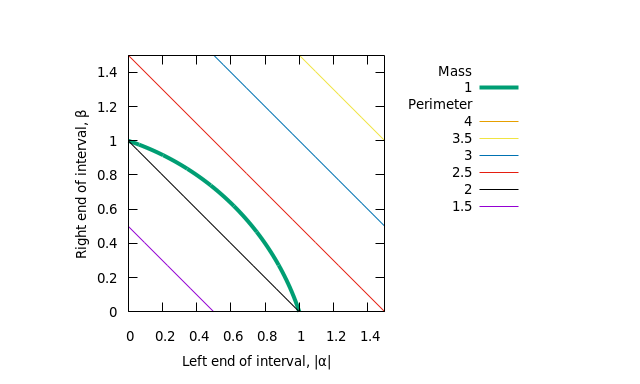}
         \caption{$p=1, a=0.5$}
         \label{fig:contourplots_n1}
     \end{subfigure}
     
     \begin{subfigure}[b]{0.48\textwidth}
         \centering
         \includegraphics[width=\textwidth]{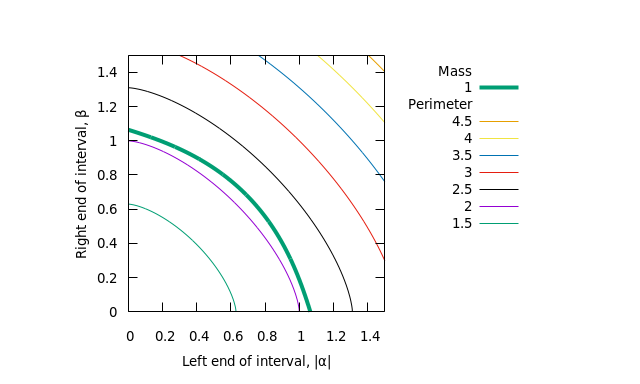}
         \caption{$p=1.5, a = 0.5$}
         \label{fig:contourplots_n1.5}
     \end{subfigure}
    \hfill
         \begin{subfigure}[b]{0.48\textwidth}
         \centering
         \includegraphics[width=\textwidth]{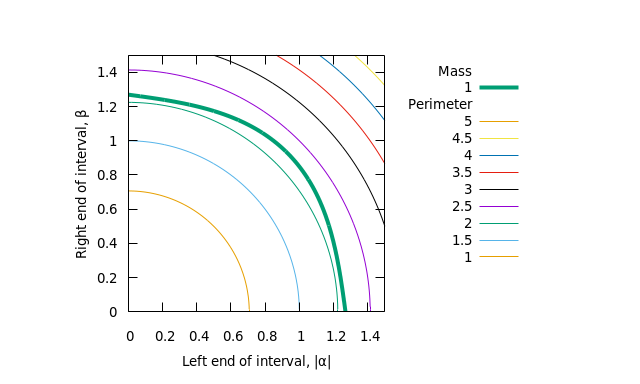}
         \caption{$p=2, a = 0.25$}
         \label{fig:contourplots_n2a}
     \end{subfigure}

     \begin{subfigure}[b]{0.48\textwidth}
         \centering
         \includegraphics[width=\textwidth]{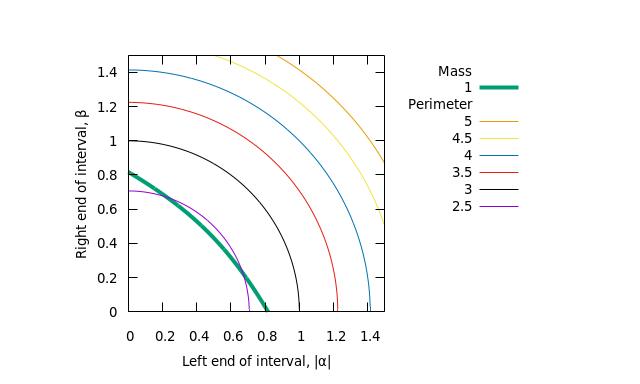}
         \caption{$p=2, a=1$}
         \label{fig:contourplots_n2b}
     \end{subfigure}
    \hfill
         \begin{subfigure}[b]{0.48\textwidth}
         \centering
         \includegraphics[width=\textwidth]{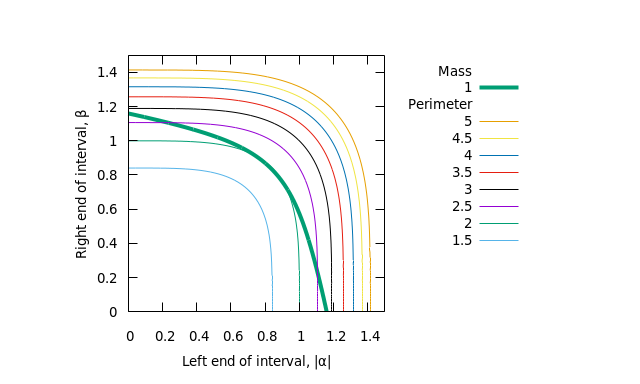}
         \caption{$p=4, a=0.2$}
         \label{fig:contourplots_n4}
     \end{subfigure}
    \caption{Contours of weighted perimeter $P(\alpha,\beta)$ for various values of $p$ and $a$, with the thick line indicating the constraint for an interval of unit weighted mass, in terms of the position of the ends of the interval.  
    }
    \label{fig:contourplots}
\end{figure}

Figure~\ref{fig:contourplots} shows numerically calculated contours of perimeter and mass, for different values of $p$ and $a$, as the endpoints $\alpha$ and $\beta$ of the interval vary. For a given mass (set equal to one in the figure), we seek the point on the mass contour with the lowest value of perimeter, from which we could read off $\alpha$ and $\beta$, or at least get an idea if the isoperimetric interval is symmetric (on the diagonal) or has one end at the origin ($\alpha = 0$ or $\beta = 0$).

For example, for $p < 1$ (figure~\ref{fig:contourplots_n0.5}), as the perimeter increases a contour of constant perimeter first meets the constraint on the axes,  indicating that the least perimeter interval has one end at the origin (as in the case $a=0$~\citep{huang2019isoperimetric}). We find that for exponents $0<p\leq 1$ the isoperimetric interval always has one end at the origin.

In the case $p=1$, the isoperimetric region again has one end at the origin. For $1<p<2$ the intersection moves in from the axes, indicating (for the value of $a$ chosen in figure~\ref{fig:contourplots_n1.5}) an asymmetric least-perimeter interval with both $\alpha$ and $\beta$ non-zero. This is a consequence of the mass contour becoming concave more rapidly with $p$ than the perimeter contours. As Lemma~\ref{lem:logconvex_allD} indicates, for exponents $p>1$ there is a critical value of $a$ above which the isoperimetric interval changes from being asymmetric, possibly with one end at the origin, to being symmetric. 
We indicate the dependence on $a$ for $p=2$ (figures~\ref{fig:contourplots_n2a} and \ref{fig:contourplots_n2b}): for small $a$ the isoperimetric interval is asymmetric, while for larger $a$ it is symmetric. That is, for large $a$ the density function becomes log-convex and Lemma~\ref{lem:logconvex_allD} applies.

For $p=\frac{1}{2}$, $p=1$ and $p=2$ we are able to derive explicit formulae for the endpoints of the isoperimetric interval $[\alpha,\beta]$ and the corresponding perimeter $P$ as $a$ varies. In the case $p=2$  we use the method of Lagrange multipliers and for $p=1$ the proof is inspired by the picture presented in figure~\ref{fig:contourplots_n1}, where the constraint on the mass can be represented as a circle. The value of $p=\frac{1}{2}$ is a special case of $p$ in the interval $(0,1)$ for which a cubic can be solved.
A summary of the way in which the perimeter and the end-points vary with the offset $a$ for $p=2$, $p=1$ and $p=\frac{1}{2}$ is shown in figure~\ref{fig:perim_ends}.
We then discuss the remaining values of $p$, for which we have been unable to find closed-form solutions.

\begin{figure}
     \centering
     \begin{subfigure}[b]{0.48\textwidth}
         \includegraphics[width=\textwidth]{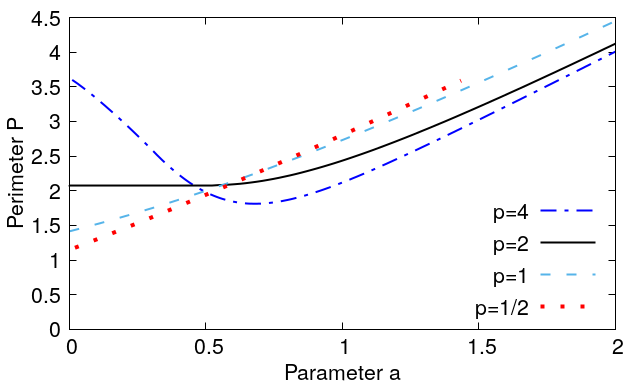}
         \caption{\quad}
         \label{fig:perim_ends1}
     \end{subfigure}
     \hfill
     \begin{subfigure}[b]{0.48\textwidth}
         \includegraphics[width=\textwidth]{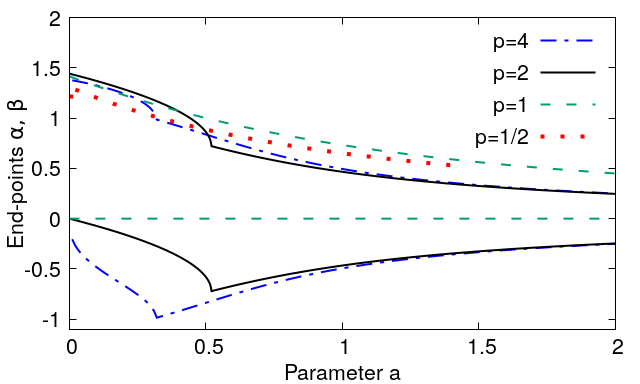}
         \caption{}
         \label{fig:perim_ends2}
     \end{subfigure}
    \caption{(a) Minimum perimeter for different values of the exponent $p$ as the parameter $a$ varies, for an interval of unit mass $M_0=1$. For $p=2$, $a_{crit}\approx 0.52$. (b) The positions of the end-points of the interval with least perimeter for different values of the exponent $p$ as the parameter $a$ varies in the case $M_0=1$. 
    In both panels the data for $p=4$ is obtained numerically, while other curves are from the equations in the text.
    }
    \label{fig:perim_ends}
\end{figure}

\subsection{The case $p=2$ in 1D}
\label{sec:resultsn2}

We use the method of Lagrange multipliers to determine the optimal solution, by writing the functional 
\begin{equation}
L(\alpha,\beta,\lambda) = P(\alpha,\beta) +
 \lambda (M(\alpha,\beta) - M_0 ) 
 \label{eq:lagrange_general}
\end{equation}
where $\lambda$ is the Lagrange multiplier and $M_0$ a target mass. We equate the derivatives of $L$ to zero in the usual way, and seek a solution for $\alpha$ and $\beta$.

\begin{lem} 
Let the density function weighting perimeter and mass be $\rho(x) =x^2+a$ for $a > 0$. The interval $[\alpha,\beta]$ of mass $M_0$ with least perimeter straddles the origin for $a < a_{crit}$, with $\alpha\beta = -a$, and is symmetric about the origin, with $\alpha=-\beta$, for $a\geq a_{crit}$. We have
\[
a_{crit} = \frac{1}{4}\left(3M_0 \right)^{2/3}
\]
and least perimeter
\[
 P = \left\{ \begin{array}{cc}
     \left(3M_0 \right)^{2/3} & a \leq a_{crit} \\
     2k+2a^2/k -2a & a \geq a_{crit},
 \end{array}\right.
\]
where $4k^{3/2} = 3M_0 +\sqrt{9M_0^2+16a^3}$. 
\end{lem}

For $a=0$ one end of the isoperimetric interval is at the origin, and as $a$ increases it moves in the negative direction until a critical value of $a$ is reached, equal to the value from the log-convex prediction, eq.~(\ref{eq:acrit_p_gt_1}), at which point the interval becomes symmetric.

\begin{proof}

In this case, eq.~(\ref{eq:perim_general}) for the perimeter becomes
\begin{equation}
 P= \alpha^2 + \beta^2 +2 a
\label{eq:perimetern2}
\end{equation}
and eq.~(\ref{eq:mass_general}) for the constraint on the mass is 
\begin{equation}
 \frac{\beta^3}{3} + \beta a -\frac{|\alpha|^3}{3} -|\alpha| a = M_0.
\label{eq:constraintn2}
\end{equation} 
We then form the Lagrange function
\begin{equation}
    L(\alpha,\beta,\lambda)=\alpha^2+\beta^2+2a+\lambda \left(\frac{\beta^3}{3}+\beta a -\frac{|\alpha|^3}{3}-|\alpha|a-M_0 \right),
    \label{eq:Lagrange_forp2}
\end{equation}
with derivatives
\begin{equation}
\frac{\partial L}{\partial \alpha}=2\alpha + \lambda ( -\alpha |\alpha| - a ),
\quad
\frac{\partial L}{\partial \beta}=2\beta + \lambda ( \beta^2 + a ).
\end{equation}
Equating these to zero and eliminating $\lambda$ gives
 \begin{equation}
|\alpha|^2\beta + |\alpha|(\beta^2 + a) + a\beta = 0
 \label{eq:n2ab}
 \end{equation}
which is a quadratic for $\alpha$ with solutions
$\alpha=-\beta$, corresponding to a symmetric solution, and $\alpha =-\displaystyle\frac{a}{\beta}$,  an asymmetric solution.

We now consider each possibility in turn to determine the least perimeter for each $a$ and $M_0$. 

{\bf Asymmetric solution:}
With $\alpha = -\displaystyle\frac{a}{\beta}$ the perimeter and the mass constraint are 
\begin{equation}
P_{asym}=\left(\frac{a}{\beta} + \beta\right)^2,
\quad \quad
  M_0 = \frac{1}{3} \left(\frac{a}{\beta} + \beta \right)^3.
    \label{eq:beta_n2}
\end{equation}
So the perimeter depends only on the length of the interval,
\begin{equation}
P_{asym}=(3M_0)^\frac{2}{3},
\label{eq:perimetern2a}
\end{equation}
and is, remarkably, independent of $a$.

To find the positions of the ends of the interval we solve the cubic for $\beta$ in eq.~(\ref{eq:beta_n2}):
\begin{equation}
    \beta=\frac{1}{2} \left[ \sqrt{(3M_0)^\frac{2}{3}-4a}+(3M_0)^\frac{1}{3} \right]
    \label{eq:n2symbeta}
\end{equation}
and then the other end of the interval is at
\begin{equation}
  \alpha=-\frac{a}{\beta}=\frac{1}{2} \left[\sqrt{(3M_0)^\frac{2}{3}-4a}-(3M_0)^\frac{1}{3}\right].
    \label{eq:n2symalpha}
\end{equation}
Note that $\beta$ is strictly positive. 
From the square root in the expression above, real solutions only exist for $a \leq a_{crit}$.

{\bf Symmetric solution:} 
With $\alpha=-\beta$ the perimeter and the constraint on the mass are
\begin{equation}
P_{sym}=2\beta^2+2a,
\quad \quad
M_0 = \frac{2\beta^3}{3}+2\beta a.
\label{eq:perimetern2_s}
\end{equation}
 The cubic can be solved to give 
\begin{equation}
  \beta=-\alpha= \biggl(\frac{3M_0}{4}+\frac{1}{4}\sqrt{9M_0^2+16a^3}\biggr)^\frac{1}{3} - \frac{a}{\biggl(\displaystyle\frac{3M_0}{4}+\frac{1}{4}\sqrt{9M_0^2+16a^3}\biggr)^\frac{1}{3}}.
    \label{eq:n2symalphabeta}
\end{equation}
Substitution of this expression in eq.~(\ref{eq:perimetern2_s}) gives the perimeter:
\begin{equation}
P_{sym}=2\left(\frac{3M_0}{4}+\frac{1}{4}\sqrt{9M_0^2+16a^3}\right)^\frac{2}{3} +
\frac{2a^2}{\left(\displaystyle\frac{3M_0}{4}+\frac{1}{4}\sqrt{9M_0^2+16a^3}\right)^\frac{2}{3}}-2a.
\label{eq:perimetern2s}
\end{equation}

\end{proof}

The perimeter and the position of the end-points of the isoperimetric intervals are plotted in figure~\ref{fig:perim_ends}. Note that the perimeter of the asymmetric and symmetric interval are equal at $a=a_{crit}$.

\subsection{The case $p = 1$ in 1D}
\label{sec:resultsn1}

We now consider the density function $\rho(x) = |x|+a$. The log-convex result, Lemma~\ref{lem:logconvex_allD}, predicts that the critical value of $a$ is infinite, and so the isoperimetric interval is never symmetric about the origin. We will show that the least-perimeter solution will be an interval with one end at the origin.

By Lemma~\ref{lem:joinintervals} we may first assume $\alpha\leq0$ and $\beta>0$. The mass  of the interval $[\alpha,\beta]$ is equal to the sum of the masses of the intervals $[0,|\alpha|]$ and $[0,\beta]$. 

\begin{lem} 
Let the density function weighting perimeter and mass be $\rho(x) =|x|+a$. The interval of mass $M_0$ with least perimeter has one end at the origin and the other at $\beta = -a+\sqrt{a^2+2M_0}$. Its perimeter is $P=a+\sqrt{a^2+2M_0}$.
\end{lem}

\begin{proof}
With $p=1$, eq.~(\ref{eq:perim_general}) for the perimeter becomes
\begin{equation}
 P= |\alpha| + \beta +2 a
\label{eq:perimetern1}
\end{equation}
and eq.~(\ref{eq:mass_general}) for the constraint on the mass becomes
\begin{equation}
  M_0 = \frac{\alpha^2}{2} + \frac{\beta^2}{2}+ a (|\alpha| + \beta) .
  \label{eq:constraintn1}
\end{equation}
For given mass $M_0$ and offset $a$ we can visualize this constraint as a circle, as shown in figure~\ref{fig:contourplots_n1}, since eq. ~(\ref{eq:constraintn1}) can be written as
\[
\alpha^2 + \beta^2 + 2a |\alpha| + 2a\beta = 2M_0
\]
or, after completing the square twice,
\[
(\beta+a)^2 + (|\alpha|+a)^2 = 2M_0 +2a^2.
\]
The circle has centre $(-a,-a)$ and radius $R=\sqrt{2(M_0+a^2)}$.

We introduce a change of variable: let $|\alpha|+a=R \cos\theta$ and $\beta+a = R\sin\theta$. The perimeter becomes $P(\theta) = R(\cos\theta+\sin\theta)$. We seek a minimum of $P(\theta)$ for $\theta \in [\pi/4,\cos^{-1}(a/R)]$, given that we have assumed $|\alpha|<\beta$. Equating ${\rm d}P/{\rm d\theta}$ to zero gives $\tan\theta=1$, and the second derivative shows that this solution, $\theta = \pi/4$, is a maximum. Therefore $P$ attains its minimum on the boundary of the domain, i.e. for $\theta = \cos^{-1}(a/R)$. Then $\alpha=0$ and one end of the isoperimetric interval is at the origin. The other end is at 
\[
\beta=-a+R\sin\theta = -a+R\sqrt{1-a^2/R^2}=-a+\sqrt{a^2+2M_0},
\]
and the perimeter, from  eq.~(\ref{eq:perimetern1}), is 
\begin{equation}
P=\beta+2a = a+\sqrt{a^2 + 2M_0}.
\label{eq:isoperim_n1}
\end{equation}

\end{proof}

These results are supported by the data in figure~\ref{fig:contourplots_n1}, where lines of constant perimeter are straight with slope -1. In effect, these (extended) lines are chords of the constraint circle. The least value of the perimeter that solves both eqs~(\ref{eq:perimetern1}) and (\ref{eq:constraintn1}) will be where the chord first intersects the circle moving from left to right, i.e. at the axes. 

The perimeter is plotted in figure~\ref{fig:perim_ends1}: $P$  increases as $a$ increases and asymptotes to a line with slope 2 as $a \to\infty$. The positions of the end-points of the corresponding optimal interval are shown in figure~\ref{fig:perim_ends2}.

\subsection{The case $0< p < 1$ in 1D}
\label{sec:resultsn01}

\begin{lem} 
Let the density function weighting perimeter and mass be $\rho(x) =|x|^p+a$ with $0 < p < 1$. The interval of weighted mass $M_0$ with least weighted perimeter has one end at the origin and the other at $x=\beta$ where $\beta$ is the solution of $\beta^{p+1}=(p+1)(M_0-a \beta)$. Its perimeter is $P=\beta^p+2a$.
\end{lem}

The proof of the lemma is based on the shape of the contours shown in figure~\ref{fig:contourplots_n0.5}. The argument works only for $p< 1$, for which the contours of constant perimeter are convex. Since the contours of constant mass are concave, they first meet when one end of the interval is at the origin.

\begin{proof}

On a contour of constant perimeter $P=P_0$, we can differentiate eq.~(\ref{eq:perim_general}) to show that 
\begin{equation}
\frac{d^2\beta}{d| \alpha |^2}= - \frac{(p-1)}{\beta^{p-1}}\left( | \alpha|^{p-2} + \frac {| \alpha |^{2n-2}}{\beta^p}\right).
     \label{eq:2ndDiffBP}
\end{equation}
This is positive for $0<p<1$, so the contour is convex.
Similarly, for a contour of constant mass $M$, we have
\begin{equation}
    \frac{d^2\beta}{d| \alpha |^2}=-\frac{1}{(\beta^p + a)^3}\left(p|\alpha|^{p-1}(\beta^p + a)^2 + p\beta^{p-1}(|\alpha|^p + a)^2\right),
     \label{eq:2ndDiffBG}
\end{equation}
which is negative for $0<p<1$,and so this contour is concave.

The contours are symmetric about the line $\beta=|\alpha|$, since the density function is symmetric about the origin. As $a$ increases, the perimeter increases and the contours move towards each other. They first intersect in the region $|\alpha| < \beta$ when $\alpha=0$. 
The other end of the interval $x=\beta$ is the solution of the constraint for the mass, eq.~(\ref{eq:mass_general}), with $\alpha=0$:
 \begin{equation}
\beta^{p+1} = (p+1)(M_0-a \beta),
   \label{eq:endpt_n_less_1}
 \end{equation}
 and then the weighted perimeter of the isoperimetric interval is 
 \begin{equation}
     P = \beta^p + 2a = (p+1)\left(\frac{M_0}{ \beta}-a\right) + 2a.
     \label{eq:perim_n_less_1}
 \end{equation}  
 \end{proof}
 
In general, numerical methods are required to explicitly calculate $\beta$, and hence $P$. An exception is the case $p=\frac{1}{2}$ for which  eq.~(\ref{eq:endpt_n_less_1}) is a cubic for $\sqrt{\beta}$ and can therefore be solved:
\[
\beta = \left (\frac{a^2}{4Z^\frac{1}{3}} + Z^\frac{1}{3}-\frac{a}{2} \right)^2
\]
where
\[
Z = \left(\frac{3M}{4}-\frac{a^3}{8}\right) -\sqrt{\left(\frac{3M}{4}-\frac{a^3}{8}\right)^2-\frac{a^6}{64}},
\]
which is valid until the discriminant becomes negative, at $a = (3M_0)^{1/3}$.
The perimeter is plotted in figure~\ref{fig:perim_ends1} in this case, and the positions of the end-points are shown in figure~\ref{fig:perim_ends2}.

\subsection{Other values of $p$ in 1D} \label{sec:resultsn>2}

For $p>2$ the log-convex result, eq.~(\ref{eq:acrit_p_gt_1}), gives the minimum value of $a$ for which the isoperimetric interval is symmetric. For smaller $a$, we expect an asymmetric interval. A numerical solution using the Surface Evolver software~\citep{brakke92}, shown in figure~\ref{fig:perim_ends}, gives an example in the case $p=4$. 
In particular, note that the perimeter decreases with increasing $a$ for the asymmetric interval. The minimum of perimeter is found at the critical value of $a$ only for $p=2$; for larger $p$ the minimum is found at larger $a$.

For values $1<p<2$ the idea is similar. There is a symmetric interval for large $a$, above the value given in eq.~(\ref{eq:acrit_p_gt_1}), and below this there is an asymmetric interval that straddles the origin. The perimeter increases monotonically with $a$, as would be expected from a visual interpolation of the curves in figure~\ref{fig:perim_ends1}.

We have not been able to derive explicit formulae for $\alpha , \beta $ and the perimeter $P$ in these cases.

For $p<0$ we have not found isoperimetric intervals for any $a>0$, just as there are no isoperimetric solutions for density $|x|^p$ for $-2\leq p<0$~\citep{carroll2008isoperimetric}.

\section{Isoperimetric solutions in the plane (2D)} \label{Ch:2D} 

We next consider two-dimensional bubbles, not necessarily circular, in the plane, with the same density function $\rho(r) = r^p+a$. Such isoperimetric regions exist~\citep{rosales2008isoperimetric}, and we assume that the  region is connected, and that it includes the origin (possibly on the periphery) since the density increases monotonically away from the origin. As in 1D, we find that we can make significant progress only for a small number of values of $p$, specifically $p=2$, where the pattern as $a$ increases of touching the origin, straddling the origin asymmetrically, and being centred on the origin is similar to the 1D case.

\begin{mydef} 
In the plane $\mathbb{R}^2$, a radially symmetric density function $\rho$ is a positive continuous function that weights each point $(r,\theta)$ in $\mathbb{R}^2$ with density $\rho(r)$.  Given a region $r(\theta) \le R(\theta)  \subset\mathbb{R}^2$, the weighted perimeter and mass of the region are given by:
\begin{equation}
    P = \int_{-\pi}^\pi {\rm d}\theta  \,\, R(\theta) \; \rho(R(\theta)) 
\end{equation}
and 
\begin{equation}
    M = \int_{-\pi}^\pi {\rm d}\theta \int_0^{R(\theta)} {\rm d}r \,  r \rho(r) .
    \label{eq:mass2d_general}
\end{equation}
\end{mydef}

In 2D, Lemma~\ref{lem:logconvex_allD} gives the critical value of $a$ above which the density function is log-convex and hence above which the isoperimetric region is a circle with centre at the origin:
\begin{equation}
    a_{crit}=\left(\frac{p+2}{3\pi p}\right)^\frac{p}{p+2} (p-1)^{-\frac{2}{p+2}} M^\frac{p}{p+2}.
    \label{eq:2Dacrit}
\end{equation}
Note that this expression is only valid for $p>1$.
In this symmetric case, the perimeter of the region is
\begin{equation}
     P=2\pi (R^{p+1}+Ra)
\label{eq:RpMax_P}
 \end{equation}
where, for given mass $M_0$, the radius $R$ of the circle is found from 
\begin{equation}
     M_0=2\pi R^2 \left(\frac{R^p}{p+2}+\frac{a}{2} \right).
\label{eq:RpMax_M}
 \end{equation}

When $a=0$ we expect the isoperimetric region to be a circle that {\em touches} the origin~\citep{dahlberg2010isoperimetric}. For intermediate values of $a$ between 0 and $a_{crit}$ the centre of the circle may shift and its radius, or even its shape, change.

If the isoperimetric region is circular, with centre at $r=r_0>0$, 
then its perimeter and mass are
\begin{equation}
P= \int_{-\pi}^{\pi} {\rm d}\theta \,\,  R \left[ (R^2+r_0^2+2R r_0 \cos\theta)^{p/2}+a\right] .
\label{eq:2Dcircle_general_P}
\end{equation}
and
\begin{equation}
M=\int_{0}^{R}  {\rm d}q \int_{-\pi}^{\pi} {\rm d}\theta \,\,  q \left[ (q^2+r_0^2+2q r_0 \cos\theta)^{p/2}+a \right].
\label{eq:2Dcircle_general_M}
\end{equation}
We can only evaluate these integrals for very few values of $p$. Plotting numerically calculated contours of perimeter and mass would again indicate the values of $r_0$ and $R$ at which the isoperimetric region might be found -- if it were circular. Even then, our examination of such contours does not suggest a novel method for approaching the problem for particular values of $p$ (as it did in 1D for $p=1$).

We examine the case $p=2$ below, using the method of Lagrange multipliers to find the perimeter of the (assumed circular) isoperimetric region for $a<a_{crit}$. There are two alternative methods that we might employ, although they turn out to be intractable for our density function. We nonetheless describe them briefly.

Let $\kappa$ be the classical (Riemannian) curvature of a plane curve $r(\theta)$, and $\underline{n}$ its normal, and again write the density as $\psi = e^{\rho(r)}$.  \cite{corwin2006differential} defines the generalised curvature:
\begin{equation}
\kappa_\psi=\kappa -\frac{d\psi}{dn}
\label{eq:kappa-gen}
\end{equation}
and proves that an isoperimetric curve has 
$\kappa_\psi$ constant, where
\begin{equation}
\kappa_{\psi}=\frac{r^2+2{\dot r}^2-r \ddot r}{(r^2+{\dot r}^2)^{3/2}}+\frac{{\rm d}\psi}{{\rm d}r}\frac{r}{{(r^2+{\dot r}^2)^{1/2}}}.
\label{eq:kappa-psi-r1}
\end{equation}

In our case, ${\rm d}\psi/{\rm d}r= p r^{p-1}/(r^p+a)$, which appears too complicated to offer solutions to this differential equation. We conjecture that $\kappa_\psi$ and the Lagrange multiplier $\lambda$ are equivalent, since $\lambda$ can be thought of as the (``weighted") bubble pressure, and pressure and curvature are the same up to a constant factor.

The second approach that we mention is due to \cite{kolesnikov2011isoperimetric}. They show that an isoperimetric region described by $\theta = f(r)$ in the plane satisfies:
\begin{equation}
\left[ \frac{r^2 f'}{\sqrt{1+r^2(f')^2}} \right]^\prime +\frac{{\rm d}\psi}{{\rm d} r} \left[\frac{r^2f'}{\sqrt{1+r^2(f^\prime)^2}}\right] = cr,
\label{eq:Kol-f2}
\end{equation}
where $c$ is a constant that we identify as the generalized curvature $\kappa_\psi$. So it appears that eqs.~(\ref{eq:kappa-psi-r1}) and (\ref{eq:Kol-f2}) are equivalent.

\subsection {The case $p=2$ in 2D}
\label{sec:2Da}

\begin{conj} 
Let the density function weighting perimeter and mass be $\rho(r) =r^2+a$ for $a > 0$. The isoperimetric solution for a single region of weighted mass  $M_0$ is a circle which straddles the origin for  $a < a_{crit}$, with centre a distance $r_0$ from the origin, and is symmetric about the origin for $a\geq a_{crit}$, where
\[
a_{crit} = \left( \frac{2M_0}{3\pi }\right)^\frac{1}{2}.
\]
The perimeter is
\begin{equation}
 P = \left\{ \begin{array}{ccc}
     4\pi\left(\displaystyle\frac{2M_0}{3\pi} \right)^{3/4} & a \leq a_{crit} & \\
     2\pi\left( -a + \sqrt{a^2+\displaystyle\frac{2M_0}{\pi}} \right)^\frac{1}{2}\sqrt{a^2+\displaystyle\frac{2M_0}{\pi}} & a \geq a_{crit}& 
 \end{array}\right.
 \label{eq:2D_p2_perimeter}
\end{equation}
and the radius is
\begin{equation}
     R = \left\{ \begin{array}{ccc}
     \left(\displaystyle\frac{2M_0}{3\pi} \right)^{1/4} & a \leq a_{crit} \\
    \left( -a + \sqrt{a^2+\displaystyle\frac{2M_0}{\pi}} \right)^\frac{1}{2} & a \geq a_{crit}& 
 \end{array}\right.
\label{eq:2D_p2_radius}
\end{equation}
with $r_0= \sqrt{R^2-a}$ for $a\le a_{crit}$.
\end{conj}

These predictions are shown in figure~\ref{fig:PR2a}, and compared with Surface Evolver simulations that do not assume a fixed circular shape. The simulations do appear to show a circle, and the agreement with the conjectured expressions above for the perimeter and radius are compelling.

To derive these predictions, we proceed as follows. In the symmetric case, we have the critical $a$ (from eq.~(\ref{eq:2Dacrit})), $a_{crit} = \sqrt{2M/3\pi} \approx 0.46$, and an expression for the mass (from eq.~(\ref{eq:RpMax_M}), $M_0=\frac{\pi}{2}(R^4 + 2R^2a)$. The latter is a quadratic for $R^2$, giving the second part of eq.~(\ref{eq:2D_p2_radius}). The perimeter in the symmetric case is given by $P=2\pi R (R^2+a)$, from eq.~(\ref{eq:RpMax_P}), which when the expression for $R$ is substituted gives the second expression in eq.~(\ref{eq:2D_p2_perimeter}).

For $a<a_{crit}$, assume that the isoperimetric region is a circle with  centre a distance $r_0 > 0$ from the origin and radius $R$. That is, the circle is not centred on the origin. Evaluating eqs.~(\ref{eq:2Dcircle_general_P}) and (\ref{eq:2Dcircle_general_M}) in this case gives
the perimeter and mass:
\[
P=2\pi (R^3+R r_0^2+Ra),
\quad \quad
M=\frac{\pi}{2} \left( R^4 + 2R^2 r_0^2 +2R^2a \right).
\] 
Then for a fixed target mass $M_0$ we form the Lagrangian $ L=P +\lambda (M-M_0)$ and set its derivatives with respect to $R$ and $r_0$ to zero to give
\begin{equation}
   3 R^2 + r_0^2 + a +\lambda R (R^2 + r_0^2 +a)=0
\quad \mbox{ and } \quad
  1 +  \frac{1}{2} \lambda R=0
   \label{eq:Lpa}
\end{equation}
and hence 
\begin{equation}
     \lambda = -\frac{2}{R}
     \quad \mbox{ and } \quad
     R^2=r_0^2+a.
   \label{eq:Pp2}
\end{equation}
Substituting for $R^2=r_0^2+a$ gives the mass $M = \frac{3}{2} \pi R^4$, which leads to the expression for $R$ given in the first part of eq.~(\ref{eq:2D_p2_radius}). Similarly, the perimeter is $P=4\pi R^3$, which gives the first part of eq.~(\ref{eq:2D_p2_perimeter}).

As confirmed in figure~\ref{fig:PR2a}, the case $p=2$ behaves in 2D much as it does in 1D: there is a critical value of $a$, dependent on the mass, below which the centre of the isoperimetric circle moves towards the origin as $a$ increases, and above which it is centred on the origin.

\begin{figure}
   \centering
    \begin{subfigure}[b]{0.45\textwidth}
         \includegraphics[width=\textwidth]{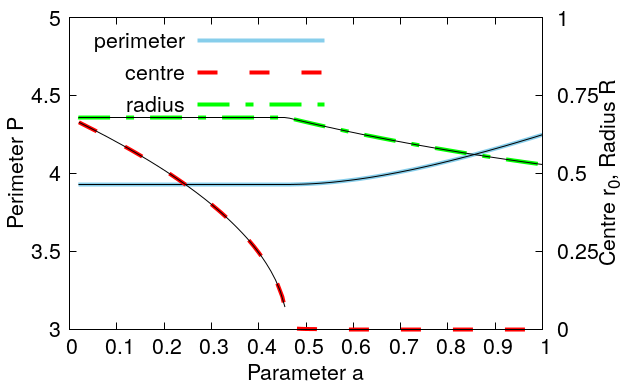}
         \caption{}
         \label{fig:PR2a}
     \end{subfigure}
     \hfill
     \begin{subfigure}[b]{0.45\textwidth}
         \includegraphics[width=\textwidth]{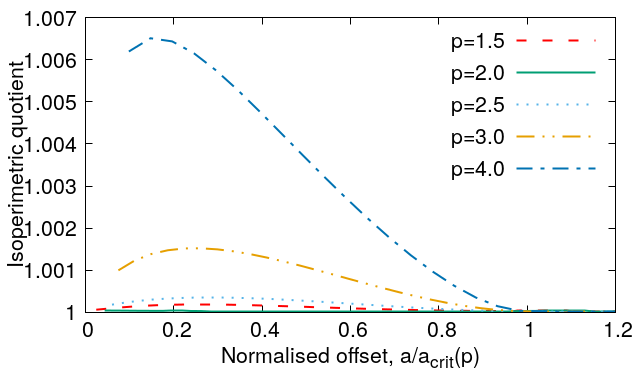}
         \caption{}
         \label{fig:PR2b}
     \end{subfigure}
    \caption{(a) Perimeter, radius and centre position of the conjectured isoperimetric circle for $p=2$, with a mass of $M=1$. The predictions from eqs.~(\ref{eq:2D_p2_perimeter}) and ~(\ref{eq:2D_p2_radius}) are shown as thin black lines and compared with the result of Surface Evolver simulations, shown as thick coloured lines.
    The agreement is excellent, suggesting that the isoperimetric solution is indeed a circle.
    (b) The isoperimetric quotient, formed from the unweighted perimeter and area of the simulated region, $P_u/\sqrt{4A_u\pi}$ {\em versus} the offset $a$ normalised by its critical value for each $p$, providing further evidence that only for $p=2$ is the isoperimetric region circular.
    }
  \label{fig:PR2}
\end{figure}

\subsection {The cases $p \ne 2$ in 2D}
\label{sec:2D_otherp}

For $p\le1$ the log-convex analysis leading to eq.~(\ref{eq:acrit_p_gt_1_allD}) is not relevant, since the density function is never log-convex (in the limit as $p$ decreases to one, $a_{crit}$ tends to infinity).
Indeed, it is not clear what the isoperimetric solutions are for $p<1$, although they are not circles.

For $p \ge 1$, numerical solutions with Surface Evolver are shown in figure~\ref{fig:2D_pics}. All figures are for the same mass -- so that the apparent size of the region increases with $p$ -- and for the same small value of $a$ -- less than the critical value of $a$ for each $p$ given in eq.~(\ref{eq:2Dacrit}). While for $p\le2$ the isoperimetric regions appear circular, we show the  isoperimetric ratio of these shapes in Figure~\ref{fig:PR2b}, which confirms that  for $p \ne 2$ they are not circular.

For $p=1$ the numerical solution suggests that the isoperimetric region is not quite circular. As $a$ increases, the approximate centre of the region moves towards the origin, but never reaches it.

For $p>2$ the isoperimetric region becomes noticeably ovoid. We have not been able to predict the shape of this region; it may be closely related to an unduloid, proposed as the boundary of an isoperimetric region in a sector with density $r^p$~\citep{diaz2012isoperimetric}.

\begin{figure}
     \centering
     \begin{subfigure}[b]{0.25\textwidth}
         \includegraphics[width=\textwidth]{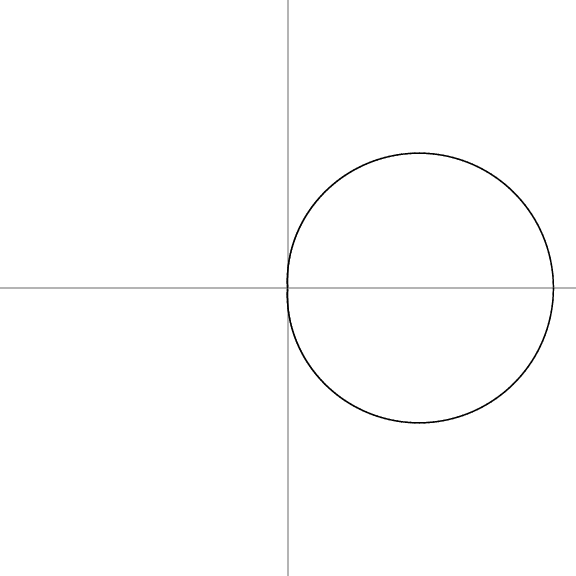}
         \caption{$p=1$}
         \label{fig:2d_pics_p1}
     \end{subfigure}
     \hfill
      \begin{subfigure}[b]{0.25\textwidth}
         \includegraphics[width=\textwidth]{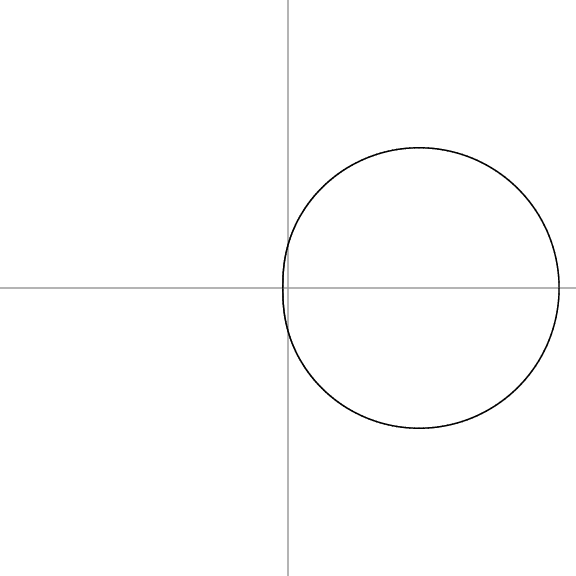}
         \caption{$p=1.5$}
         \label{fig:2d_pics_p1.5}
     \end{subfigure}
     \hfill
     \begin{subfigure}[b]{0.25\textwidth}
         \includegraphics[width=\textwidth]{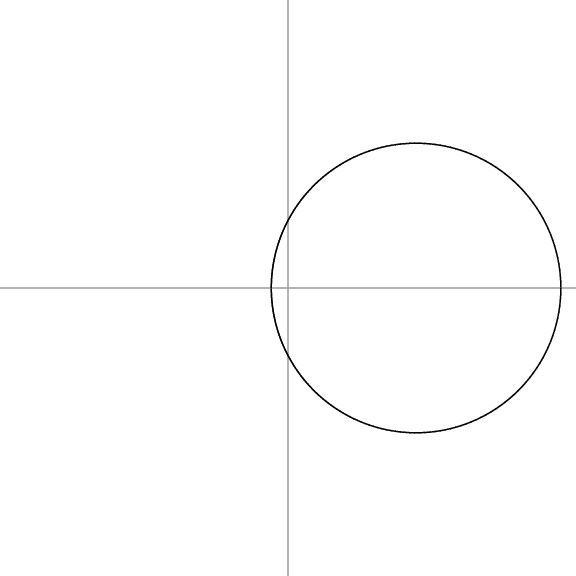}
         \caption{$p=2$}
         \label{fig:2d_pics_p2}
     \end{subfigure}

     \begin{subfigure}[b]{0.25\textwidth}
         \includegraphics[width=\textwidth]{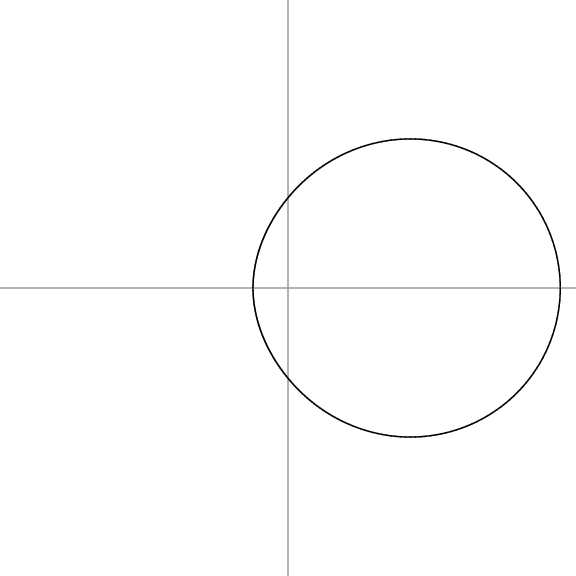}
         \caption{$p=2.5$}
         \label{fig:2d_pics_p2.5}
     \end{subfigure}
     \hfill
     \begin{subfigure}[b]{0.25\textwidth}
         \includegraphics[width=\textwidth]{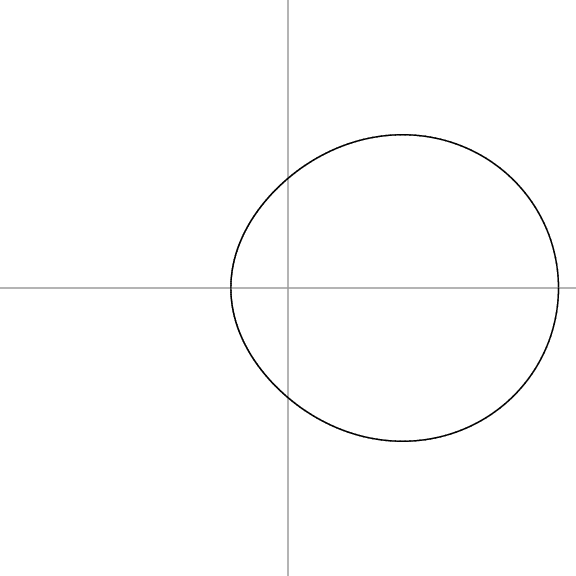}
         \caption{$p=3$}
         \label{fig:2d_pics_p3}
     \end{subfigure}
     \hfill
     \begin{subfigure}[b]{0.25\textwidth}
         \includegraphics[width=\textwidth]{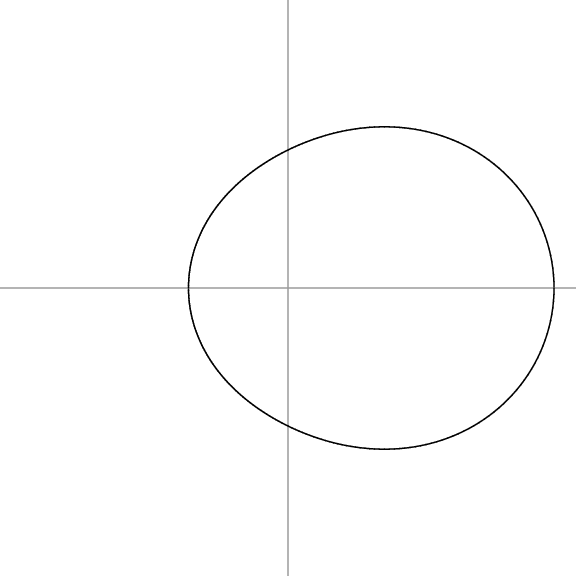}
         \caption{$p=4$}
         \label{fig:2d_pics_p4}
     \end{subfigure}
     \hfill

    \caption{Isoperimetric regions in 2D calculated wth Surface Evolver for several values of $p$. The offset is $a=0.1$, and the mass is $M=1$.
    The axis of symmetry, and the centre of the region, is constrained to lie on the $x$-axis, without loss of generality.
    }
  \label{fig:2D_pics}
\end{figure}

\section{Isoperimetric solutions in 3D}  
\label{Ch:3D} 

We make some brief remarks about the 3D case. In the case $p=2$, the same progression of a ball moving from touching the origin for $a=0$ to a ball centred on the origin above a critical value of $a$ appears to hold. And again, we have not been able to show this rigorously.

Lemma~\ref{lem:logconvex_allD} indicates that (for $p>1$) the density is log-convex and isoperimetric regions are spheres, with centre at the origin, when $a$ is greater than
\begin{equation}
 a_{crit} 
  =
  \left(\frac{3(p+3)}{16 \pi p}\right)^\frac{p}{p+3}  
  (p-1)^{-\frac{3}{p+3}}
  M^\frac{p}{p+3},
 \label{eq:aChCrit3D}
\end{equation}
which is again only valid for $p>1$.

The surface area of the sphere is
\begin{equation}
     S=4\pi (R^{p+2}+R^2a)
\label{eq:RpMax_P3D}
 \end{equation}
where, for given mass $M_0$, the radius $R$ of the circle is found from 
\begin{equation}
     M_0=4\pi R^3 
     \left(\frac{R^{p+3}}{p+3}+\frac{a}{3} \right).
\label{eq:RpMax_M3D}
 \end{equation}
It appears straightforward to find $R$ in the case $p=3$, but not in the case $p=2$, for example.

Numerical experiments suggest that for $a < a_{crit}$ only in the case $p=2$ is the isoperimetric region still spherical, and in this case it is not centred on the origin. 
As in 2D, we are not able to make progress for $p \ne 2$.

\subsection {The case $p=2$ in 3D}
\label{sec:3Dp2}

In the case $p=2$, eq.~(\ref{eq:aChCrit3D}) gives $\displaystyle a_{crit}=\left(\frac{15M}{32\pi}\right)^\frac{2}{5}$. As noted above, we are unable to find a closed form expression for $R$ from eq.~(\ref{eq:RpMax_M3D}) for a sphere centred on the origin in this case.

\begin{figure}
   \centering
         \includegraphics[width=0.5\textwidth]{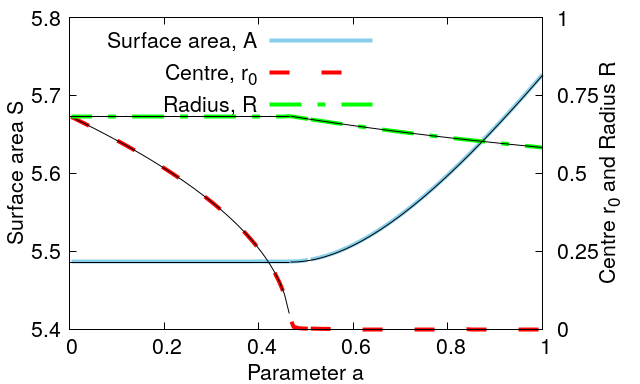}
    \caption{Surface area, radius and centre position of the conjectured isoperimetric sphere for $p=2$, with a mass of $M=1$. The predictions are shown as thin black lines: eq.~(\ref{eq:PandM3Dasym}) for $a<a_{crit}$ and for $a > a_{crit}$ these are the result of plotting eqs.~(\ref{eq:RpMax_P3D}) and (\ref{eq:RpMax_M3D}) with $R$ as a free parameter. These are compared with the result of Surface Evolver simulations, shown as thick coloured lines.
    The agreement is excellent, suggesting that the isoperimetric solution is indeed a sphere.
    }
  \label{fig:PR3}
\end{figure}

For $a<a_{crit}$, we assume that the isoperimetric region is a sphere with centre a distance $r_0 > 0$ from the origin and radius $R$. That is, the sphere is not centred on the origin. The mass and surface area are
\[
M=\int_{0}^{R} {\rm d}q \int_{-\pi}^{\pi} {\rm d}\theta  \int_{0}^{\pi} {\rm d}\phi \,\, q^2 \sin\phi \left[ (q^2+r_0^2+2qr_0 \sin\phi \cos\theta)^{p/2}+a \right]
\]
and
\[
S= \int_{-\pi}^{\pi} {\rm d}\theta \int_{0}^{\pi} {\rm d}\phi \,\, R^2\sin\phi   \left[ (R^2+r_0^2+2Rr_0 \sin\phi \cos\theta)^{p/2}+a\right],
\]
which can be evaluated in the case $p=2$ to give
\[
A=4\pi (R^4+R^2 r_0^2+R^2 a),
\quad \quad
M=\frac{4\pi}{15} \left( 3R^5 + 5 R^3 r_0^2 +5 R^3a \right).
\] 
Then for a fixed target mass $M_0$ we form the Lagrangian as in 2D, obtaining
\begin{equation}
     \lambda = -\frac{3}{R}
     \quad \mbox{ and } \quad
     R^2=r_0^2+a,
   \label{eq:Pp3}
\end{equation}
where the second of these is identical to the 2D result and the first replaces $2$ with $3$, suggesting that it is a consequence of changing dimension.
Substituting $r_0^2+a=R^2$ gives the mass and surface area:
\begin{equation}
M = \frac{32}{15} \pi R^5
\quad \quad
P=8\pi R^4 = 8 \pi \left(\frac{15 M}{32\pi}\right)^\frac{4}{5}.
\label{eq:PandM3Dasym}
\end{equation}
So both the radius and surface area of the isoperimetric region are independent of $a$ below the critical value.

This suggests that (for given mass) as $a$ increases the isoperimetric solution is a sphere of constant radius which migrates towards the origin. Once $a_{crit}$ is reached, the centre of the sphere remains at the origin and its radius decreases. These predictions are compared with numerical simulations in figure~\ref{fig:PR3}, showing excellent agreement. This is the same pattern seen in both 1D and 2D for $p=2$, and it is possible that it continues to higher dimensions.

\section{Conclusions}

We have studied the shape and location of the isoperimetric region in a space with density $\rho(r) = r^p+a$, extending work on the density $r^p$.

In the case $p=2$, in dimensions one to three, as $a$ increases the isoperimetric region migrates from an asymmetric position to a symmetric one, with constant size. The symmetric state is realised at a critical value of $a$ dependent on $p$ and the mass $M$ of the interval. Above this threshold, the region remains centred on the origin and shrinks as $a$ increases further. 

We have given a proof of this in one dimension, given compelling evidence in 2D, and outlined what happens in 3D. A similar pattern is found in 1D for any $p>1$, but for dimensions two and three the isoperimetric region becomes non-circular/non-spherical for $p \ne 2$.

The three-dimensional result may be susceptible to comparison with experiment. Consider the experiment in which a hemispherical bubble is formed on a horizontal surface. Bring a second surface, spherical in shape (such as the lower part of a round bowl) down on to the bubble and squeeze it until the bowl touches the plane. The bubble moves to one side and its perimeter is pinned  at the contact point. As the bowl is gradually raised, the bubble moves back towards the centre of the bowl. This experiment therefore provides a physical manifestation of our 3D result in the case $p\approx 2$. We have not experimented with different shaped bowls.

Several results for the density $\rho(x) = |x|^p$ appear to carry over to the case $\rho(x) = |x|^p+a$, for example that a (1D) double bubble (i.e. two separate masses) is two adjacent intervals that meet at the origin~\citep{huang2019isoperimetric}. It would be interesting to determine whether there are analogous results for 2D double bubbles with our density function, and whether there are further more general results that carry over to this case. It  may be possible to show that the isoperimetric regions are symmetric using tools from geometric measure theory~\citep{morgan2016geometric}, and the question of connectedness still remains open.

\appendix
\section{Reduction of many intervals to one interval that includes the origin: proof of Lemma~\ref{lem:joinintervals}}
\label{sec:app-twointervals}

We consider a density function $\rho(x)$ on the real line, with $\rho$ monotonically increasing on $x>0$ and monotonically decreasing on $x<0$.

We will show that a finite set of intervals with a given total mass can be reduced to a single interval  of the same mass with an equal or lower perimeter $P$. We will then show that this interval has lower perimeter if it includes the origin, by first proving in Lemma~\ref{lem:moveinterval1} that the perimeter is reduced if it has one end at the origin, and then in Lemma~\ref{lem:concatenateplusminus} proving that this is still the case if the interval includes the origin. 

We define $F$ to be the primitive of $\rho$:
\[
F(q)=\int_{0}^{q} \rho(x) \,{\rm d}x.
\]  
Then where $\rho(x)$ is monotonically increasing, $F(q)$ is monotonically increasing, and {\it vice versa}.

In Lemma~\ref{lem:joininterval1_sub} we prove that the perimeter of two non-overlapping intervals in the positive half-line is reduced if they are concatenated. If we repeatedly apply this lemma to any finite set of intervals $[\alpha_i,\beta_i] , 1\leq i\leq N,$,  taking them two at a time, we can concatenate any number of intervals to a single interval with the same total mass.

\begin{lem} 
For monotone increasing density $\rho(x)$, the perimeter of two non-overlapping intervals in the positive half-line is reduced if they are concatenated. 
\label{lem:joininterval1_sub}
\end{lem}

\begin{proof}

Consider two non-overlapping intervals $[\alpha_1,\beta_1]$ and  $[\alpha_2,\beta_2]$ in the positive half-line with $0<\alpha_1<\beta_1<\alpha_2<\beta_2$  and total non-zero mass $M_0= M_1 + M_2$ where:
\begin{equation}
    M_1=\int_{\alpha_1}^{\beta_1} \rho(x) \,dx = F(\beta_1) - F(\alpha_1) 
\end{equation} and 
\begin{equation}
  M_2 = \int_{\alpha_2}^{\beta_2} \rho(x) \,dx = F(\beta_2) - F(\alpha_2)
\end{equation}
and perimeter 
\[
P_{init}=\rho(\alpha_1) + \rho(\beta_1) + \rho(\alpha_2) + \rho(\beta_2).
\]
As $\alpha_2>\beta_1$ we can move the interval $[\alpha_2,\beta_2]$ to the left to an interval $[\beta_1,\beta_3]$, where the lower end of the interval coincides with the upper end of the other interval and $\beta_3$ is chosen to keep the mass $M_2$ unchanged. We have $\beta_3<\beta_2$ since $\rho$ is monotone increasing.

The new perimeter is 
\[
P_{new} = \rho(\alpha_1) + \rho(\beta_1) + \rho(\beta_1) + \rho(\beta_3)
\]
which is less than $P_{init}$ because $\rho(\beta_1)<\rho(\alpha_2)$ and $\rho(\beta_3)<\rho(\beta_2)$.

We therefore have an interval $[\alpha_1,\beta_3]$ with the same mass as the original two intervals and reduced perimeter. 

\end{proof}

\begin{lem}
With monotone increasing density $\rho(x)$, the perimeter of a single interval $[\alpha,\beta]$ in the positive half-line  ($0< \alpha <\beta$) is reduced if it is moved so that it has one end at the origin.
\label{lem:moveinterval1}
\end{lem}

\begin{proof}

Consider an interval $[\alpha_1,\beta_1]$ with $0<\alpha_1<\beta_1$  and non-zero mass 
\begin{equation} 
M_0=\int_{\alpha_1}^{\beta_1} \rho(x) \,dx = F(\beta_1) - F(\alpha_1).
    \label{eq:moveintervalmass1}
\end{equation} 
The perimeter is
\begin{equation} 
P_{init}=\rho(\alpha_1) + \rho(\beta_1).
\end{equation}

Move the interval left to $[0,\beta_2]$, then
\begin{equation}
    M_0=\int_{0}^{\beta_2} \rho(x) \,dx = F(\beta_2) - F(0)
    \label{eq:moveintervalmass2}
\end{equation}
where $\beta_2$ is chosen to leave mass unchanged.
The perimeter is now 
\begin{equation}
  P_{new}=\rho(0) + \rho(\beta_2) . 
\end{equation} 
We will show that $P_{new}<P_{init}$.

Equating the expressions for the mass, eqs~(\ref{eq:moveintervalmass1}) and (\ref{eq:moveintervalmass2}), gives:
\[
F(\beta_2) - F(0) = F(\beta_1) - F(\alpha_1).
\] 
Rearranging gives
\[
F(\beta_1) - F(\beta_2) = F(\alpha_1) - F(0).
\] 
Now $F(\alpha_1) - F(0)>0$ because $F$ is monotonically increasing. So $F(\beta_1) - F(\beta_2) > 0$ and hence  $\beta_1 > \beta_2$. Then $\rho(\beta_1) - \rho(\beta_2)>0$ because $\rho$ is monotonically increasing. 
Also, as $\alpha_1>0$, $\rho(\alpha_1)-\rho(0)>0$.
Then
\[
P_{new}-P_{init} = 
 [ \rho(0) + \rho(\beta_2)] - [\rho(\beta_1) + \rho(\alpha_1)] = 
  -[\rho(\beta_1) - \rho(\beta_2)] - [\rho(\alpha_1) - \rho(0)].
\] 
Both terms are negative, so $P_{new} < P_{init}$ as required.
Therefore  $[0,\beta_2]$ has a smaller perimeter than $[\alpha_1,\beta_1]$.

\end{proof}

\begin{lem}
For any density function $\rho(x)$ that is non-zero at the origin, the perimeter of two non-overlapping intervals, one in the positive half-line and the other in the negative half-line, each with one end at the origin, is reduced if they are concatenated to form a single interval that contains the origin. 
\label{lem:concatenateplusminus} 
\end{lem}

\begin{proof}
The perimeter of two intervals $[\alpha,0]$ and $[0,\beta]$, with $\alpha < 0 $ and $\beta>0$, is $P=\rho(\alpha) + \rho(0) + \rho(\beta) + \rho(0) $. 

If we merge these two intervals to create an interval $[\alpha,\beta]$, it clearly has the same mass, but its perimeter will be reduced by $2\rho(0)$.
\end{proof}

\bibliographystyle{plainnat}

\end{document}